\documentclass[11pt]{article}

\usepackage[a4paper,margin=1in]{geometry}
\usepackage{amsmath,amssymb,amsthm,mathtools}
\usepackage{enumitem}
\usepackage{microtype} %improve readability
\usepackage[colorlinks]{hyperref} %colors to references
\usepackage{dsfont} %1-indicators

\newtheorem{theorem}{Theorem}[section]
\newtheorem{proposition}[theorem]{Proposition}
\newtheorem{lemma}[theorem]{Lemma}

\theoremstyle{remark}
\newtheorem{remark}[theorem]{Remark}

\newcommand{\N}{\mathbb{N}}
\newcommand{\Pp}{\mathbb{P}}
\newcommand{\E}{\mathbb{E}}

\title{Collision Properties in Discrete and Continuous Time on Bounded-Degree Graphs}
\author{Jhon Astoquillca \thanks{Institute of Mathematics, Statistics and Computer Science, University of S\~ao Paulo, \url{jhonastoquillca@ime.usp.br} }}
\date{}

\begin{document}

\maketitle

\begin{abstract}
We prove that, on every connected bounded-degree graph, the infinite collision property holds in discrete time if and only if it holds in constant-speed continuous time, and the same equivalence holds for the finite collision property. The proof is based on a pointwise comparison between the Green kernels of the synchronous discrete-time pair chain and the asynchronous pair chain obtained by Poissonization. The main estimate exploits the binomial interlacing of the coordinate updates together with a local binomial estimate. We then use Martin capacities and collision zero--one laws to relate this Green-kernel comparison to infinite visits of the diagonal.
\end{abstract}

\section{Introduction}
Let $G=(V,E)$ be an infinite connected graph, and let~$(X_n)_{n\geq 0}$ and~$(Y_n)_{n\geq 0}$ be two independent simple random walks on $G$. A collision occurs at time~$n$ whenever~$X_n=Y_n$. We say that~$G$ has the discrete infinite collision property if, when the two walks start from the same vertex, they collide infinitely often almost surely; when the number of collisions is almost surely finite, we say that~$G$ has the discrete finite collision property. The analogous notions for continuous-time simple random walks are defined in the natural way; the precise definition is given in~\eqref{def_collision_event_cont}. Throughout, we consider the constant-speed version: at each vertex, the walk waits an exponential time of rate $1$ and then takes one simple-random-walk step.

Collision properties arise naturally in systems of random walks and interacting particle systems. Through the classical duality between the voter model and coalescing random walks, they play an important role in the description of stationary measures of voter-type systems~\cite{Liggett85,Astoquillca26}. This makes it useful to understand which aspects of collision behavior are intrinsic to the spatial motion and which depend on the chosen parametrization of time. In discrete time the two walkers move simultaneously, whereas in continuous time only one walker moves at each jump time of the joint process. Thus, although the two models are built from the same underlying random-walk trajectories, their collision behavior need not a priori be the same.

This motivates the question of whether collision behavior depends on the parametrization of time. The issue already appears naturally in the study of random walks on random media. For instance, Chen and Chen~\cite{ChenChen10} proved the infinite collision property for continuous-time simple random walks on the supercritical open cluster of bond percolation in $\mathbb Z^2$, using the heat-kernel estimates of Barlow~\cite{Barlow04}, which were established in continuous time. They also raised the corresponding discrete-time question, noting that Barlow stated that analogous estimates should hold in discrete time. More generally, Chen, Wei and Zhang~\cite{Chen08} had shown that the infinite collision property is independent of the time parametrization for quasi-transitive graphs of subexponential growth, and this growth assumption was subsequently removed in the quasi-transitive setting; see~\cite[Theorem 3.37 and 3.38]{MontgomeryThesis}. On the other hand, the time parameter cannot be ignored on arbitrary graphs: Montgomery constructed a transient graph of unbounded degree for which the discrete- and continuous-time collision behaviors differ~\cite[Claim~3.9]{MontgomeryThesis}.

Our main result gives an affirmative answer to the question raised in~\cite{AstoMar2026} of whether discrete- and continuous-time collision properties coincide on bounded-degree connected graphs.

\begin{theorem}\label{thm:main}
Let $G$ be a bounded-degree connected graph. Then:
\begin{enumerate}
    \item $G$ has the discrete infinite collision property if and only if it has the continuous infinite collision property, 
    \item $G$ has the discrete finite collision property if and only if it has the continuous finite collision property.
\end{enumerate}

\end{theorem}

\subsection{Organization of the paper and proof strategies}
Our approach is based on a comparison between the Green kernels of the joint process. The particular form of continuous time pair provides a useful binomial representation. This representation allows us to compare the Green kernel of the continuous-time skeleton with the Green kernel associated with the synchronous discrete-time pair. One direction of this comparison follows directly from the binomial representation. The reverse direction is more delicate: retaining only the terms in which the two coordinates have made the same number of jumps loses a factor of order $\sqrt n$. To recover this loss, we show that enough transition mass is distributed over a suitable window of nearby times. This yields the Green-kernel comparison needed for the collision argument.

To turn these potential-theoretic estimates into statements about infinitely many collisions, we use the Martin-capacity estimates of Benjamini, Pemantle and Peres~\cite{BenjaminiPemantlePeres1995}. Their theorem relates the probability that a transient Markov chain visits an arbitrary subset infinitely often to the asymptotic Martin capacity of that subset. Applying this criterion to the diagonal $\Delta$, together with the Green-kernel comparison and appropriate zero--one laws for collisions, allows us to transfer the infinite collision property between discrete- and continuous-time random walks.

The remainder of the paper is organized as follows. In Section~\ref{ss_MarkovChain}, we develop the Markov-chain tools needed for the proof. We first establish zero--one laws for discrete- and continuous-time collision events, and then recall the Martin-capacity criterion of Benjamini, Pemantle and Peres and use it to show that a pointwise comparison of the Green kernels of the two pair chains yields a comparison of their infinite collision probabilities. We then prove Theorem~\ref{thm:main} by combining this criterion with a Green-kernel comparison for the synchronous and asynchronous pair chains. The proof of this comparison, which is the main technical ingredient of the argument, is given separately in Section~\ref{ss_lem:green-comparison} and relies on local binomial estimates. Finally, in Section~\ref{ss_further} we discuss the limitations of the method for more general time parametrizations, in particular variable-speed random walk.

\section{Preliminaries and Markov chain results}\label{ss_MarkovChain}
Given a set or event~$A$, we denote by~$\mathds{1}_A$ its indicator function and by~$|A|$ its cardinality. Throughout this section, let~$\mathsf S$ be a nonempty countable state space.

\subsection{0-1 law for reversible Markov chains}
Let~$P$ be an irreducible Markov kernel on~$\mathsf S$, and let~$(X_n)_{n \ge 0}$ and~$(Y_n)_{n \ge 0}$ be two independent copies of the corresponding discrete-time Markov chain. We define the \emph{discrete infinite collision event},
\begin{equation}\label{def_collision_event_disc}
\mathcal E^P_\mathrm{disc} := \big\{ |\{n \ge 0: X_n=Y_n  \}| = \infty \big\}
\end{equation}
We say that~$Q = \{q(x,y): x,y \in \mathsf  S\}$ is a~$Q$-matrix if the off-diagonal entries~$q(x,y), x \neq y$, are non-negative real numbers and
$$ q(x,x)=-\sum_{y\in \mathsf S \setminus \{x\}}q(x,y), \quad
q(x):=-q(x,x)<\infty, \quad \text{for all } x \in \mathsf S. $$
Here~$q(x)$ denotes the total jump rate from~$x$. Let~$Q$ be an irreducible~$Q$-matrix on~$\mathsf S$, and let~$(X_t^{\mathrm{cont}})_{t\geq0}$ and~$(Y_t^{\mathrm{cont}})_{t\geq0}$ be two independent copies of the corresponding continuous-time Markov chain. We define the \emph{continuous infinite collision event},
\begin{equation}\label{def_collision_event_cont}
\mathcal E^Q_\mathrm{cont} := \big\{ |\{t \ge 0: X^\mathrm{cont}_{t^-} \neq  Y^\mathrm{cont}_{t^-}\text{ and } X^\mathrm{cont}_t= Y^\mathrm{cont}_t\}| = \infty \big\}.
\end{equation}

\begin{proposition}\label{prop_0-1_law}
The following statements hold.
\begin{enumerate}
    \item Let~$P$ be an irreducible Markov kernel on $\mathsf S$, reversible with respect to a measure $\pi$ satisfying
    \[ \sup_{x \in \mathsf S}\pi(x) <\infty.\]
    Then~$\Pp(\mathcal E^P_\mathrm{disc}) \in \{0,1\}$ for every common starting state.  
    \item Let~$Q$ be an irreducible~$Q$-matrix on~$\mathsf S$, reversible with respect to a measure~$\pi$ satisfying
    \[ \sup_{x \in \mathsf S} \pi(x) <\infty \qquad \text{and} \qquad \sup_{x \in \mathsf S}q(x) <\infty .\]
    Then~$\Pp(\mathcal E^Q_\mathrm{cont}) \in \{0,1\}$ for every common starting state. 
\end{enumerate}
\end{proposition}
\begin{proof}
We distinguish according to whether the Markov chains are transient or recurrent. Fix the initial state~$\rho \in \mathsf S$, by irreducibility~$\pi(\rho)>0$.

\paragraph{Assume the~$P$-chain and the~$Q$-chain are transient.} Write~$p_n(x,y)=P^n(x,y)$,~$x,y \in \mathsf S$. By reversibility,
\begin{equation}\label{eq_reversibility}
p_{2n}(\rho,\rho) = \sum_{y\in \mathsf S}p_n(\rho,y)p_n(y,\rho) = \pi(\rho)\sum_{y\in \mathsf S}\frac{p_n(\rho,y)^2}{\pi(y)}.
\end{equation}
Then, by Tonelli's theorem,
\[ \E \left[ \sum_{n\geq0}\mathds{1} \{X_n=Y_n\} \right] = \sum_{n \ge 0} \sum_{y \in \mathsf S} p_n(\rho,y)^2 \le \frac{ \sup_{x \in \mathsf S}\pi(x) }{\pi(\rho)} \cdot \sum_{n \ge 0} p_{2n}(\rho,\rho)<\infty. \]
Hence~$\Pp(\mathcal E^P_\mathrm{disc}) = 0$. We argue similarly in continuous time. Let~$q_t$ denote the transition probabilities of the~$Q$-chain. By reversibility, the analogue of~\eqref{eq_reversibility} holds with~$p_n$ and~$p_{2n}$ replaced by~$q_t$ and~$q_{2t}$. Hence, by Tonelli's theorem,
\begin{equation}\label{ineq_cont_meet}
\E \left[ \int^\infty_0 \mathds{1}\{ X^\mathrm{cont}_t = Y^\mathrm{cont}_t \} \; \mathrm{d}t \right] = \int^\infty_0 \sum_{y \in \mathsf S} q_t(\rho,y)^2 \; \mathrm{d}t \le \frac{ \sup_{x \in \mathsf S}\pi(x) }{\pi(\rho)} \cdot \int^\infty_0 q_{2t}(\rho,\rho) \; \mathrm{d}t < \infty.     
\end{equation}
We define~$\mathcal T_\mathrm{meet} := \{ t \ge 0: X^\mathrm{cont}_{t^-} \neq Y^\mathrm{cont}_{t^-} \text{ and } X^\mathrm{cont}_t = Y^\mathrm{cont}_t \}$. For each~$t \in \mathcal T_\mathrm{meet}$, denote by~$x_t \in \mathsf S$ the meeting site~$x_t = X^\mathrm{cont}_t = Y^\mathrm{cont}_t$. On the event~$\mathcal E^Q_\mathrm{cont}$, we have~$|\mathcal T_\mathrm{meet}| = \infty$. Let~$E_t$ be the holding time of the pair chain at~$x_t$, then, we can write
$$ \int^\infty_0 \mathds{1}\{X^\mathrm{cont}_t  = Y^\mathrm{cont}_t\} \; \mathrm{d}t \ge \sum_{t \in \mathcal T_\mathrm{meet}} E_t, \qquad E_t \sim \mathrm{Exp}( 2q(x_t) ). $$
Conditionally on the sequence~$(x_t)_{t \in \mathcal T_\mathrm{meet}}$, the variables~$(E_t)_{t \in \mathcal T_\mathrm{meet}}$ are independent. Then, for every~$t \in \mathcal T_\mathrm{meet}$,
$$ \Pp \left. \left( E_t \ge \frac{1}{2 \cdot \sup_{x \in S}q(x)} \right| (x_t)_{t \in \mathcal T_\mathrm{meet}} \right) \ge \exp(-1). $$
Hence, by the Borel–Cantelli lemma,~$\sum_{t \in \mathcal T_\mathrm{meet}} E_t = \infty$. Thus
\[ \mathcal E^Q_\mathrm{cont} \subseteq \left\{ \int^\infty_0 \mathds{1}\{X^\mathrm{cont}_t = Y^\mathrm{cont}_t \} =\infty \right\} \qquad\text{a.s.} \]
Finally, from~\eqref{ineq_cont_meet} it follows that~$\Pp( \mathcal E^Q_\mathrm{cont} ) = 0$.

\paragraph{Assume the~$P$-chain and the~$Q$-chain are recurrent.} The discrete-time statement follows from Proposition 2.1 and Remark 2.2 of~\cite{BarlowPeresSousi2012}: their argument applies to any recurrent Markov chain. We briefly recall the mechanism, since the same argument will be used in continuous time. If
\[ \mathcal T^X_n := \sigma(X_k:k\geq n) \quad \text{ and } \quad \mathcal T^X = \bigcap_{n\geq0} \mathcal T^X_n, \]
recurrence and irreducibility imply, by Orey's theorem, that~$\mathcal T^X$ is trivial, this is called the tail~$\sigma$-algebra of~$X$. The same is true for~$(Y_n)_{n \ge 0}$. By independence,
\[\bigcap_{n\geq0} \sigma( \mathcal T^X_n, \mathcal T^Y_n )  = \sigma (\mathcal T^X, \mathcal T^Y) \]
up to null sets. Since~$\mathcal E^P_\mathrm{disc}$ belongs to the~$\sigma$-algebra on the left-hand side, it follows that~$\Pp(\mathcal E^P_\mathrm{disc})\in \{0,1\}.$

Following the discrete-time argument, it suffices to prove a continuous-time version of Orey's theorem, namely, that the tail~$\sigma$-algebra
$$ \mathcal T^X_\mathrm{cont} := \bigcap_{t \ge 0} \sigma( Z_s : s \ge t)$$
of a recurrent chain~$(Z_t)_{t \ge 0}$ is trivial. To do so, we assume~$Z_0 = \rho$ and define the successive return times~$0=\tau_0 < \tau_1 < \tau_2 < \cdots,$ by
\[ \tau_{k+1} := \inf\{t>\tau_k : Z_t=\rho,\ Z_s\neq \rho \text{ for some } s\in(\tau_k,t)\}, \quad k \in \N \cup \{0\}. \]
On the Skorokhod space~$D([0,\infty),\mathsf S)$ we define for each~$k$ the padded excursion~$E_k \in D([0,\infty),\mathsf S)$ by
$$ E_k(t)= \begin{cases}
Z_{\tau_{k-1}+t}, & 0\le t \le \tau_k-\tau_{k-1},\\
\rho, & t > \tau_k-\tau_{k-1}.
\end{cases} $$
Since the chain is recurrent and nonexplosive, every excursion has finite duration and~$\tau_k \uparrow \infty$ almost surely. A finite permutation of the excursion sequence therefore changes the trajectory only on a bounded time interval: after all the permuted excursions have been completed, the original and permuted trajectories coincide. Consequently, every event in~$ \mathcal T^X_\mathrm{cont}$, is invariant under finite permutations of the excursion sequence. Moreover, by the strong Markov property, the sequence $(E_k)_{k\ge1}$ is an i.i.d. sequence in~$D([0,\infty),\mathsf S)$. Hence, by the Hewitt--Savage $0$--$1$ law~\cite[Theorem~3.15]{Kallenberg},~$\mathcal T^X_\mathrm{cont}$ is trivial. This completes the proof.
\end{proof}

\begin{remark}\label{remark_ICP_propagation}
By irreducibility of the chain and the Markov property applied to the pair process, the event~$\mathcal E^P_\mathrm{disc}$ propagates to all initial common states, that is, if~$\Pp( \mathcal E^P_\mathrm{disc} ) = 1$ for some initial condition~$X_0=Y_0=x$, then~$\Pp( \mathcal E^P_\mathrm{disc} ) = 1$ for any other initial state~$X_0=Y_0=x'$. The same holds for the event~$\mathcal E^Q_\mathrm{cont}$.
\end{remark}
\begin{remark}\label{remark_dichotomy_ICPFCP}
Proposition~\ref{prop_0-1_law} implies the following dichotomy: either the two Markov chains (discrete-time or continuous-time) collide infinitely many times almost surely, or they collide only finitely many times almost surely.    
\end{remark}

As applications, we consider the following two cases
\begin{enumerate}
\item Let~$G$ be a bounded degree, connected simple graph. Consider the~$Q$-matrix
$$q(x,y) = \deg(x)^{-1} \cdot \mathds{1}\{ \{x,y\} \in E(G)\}, \quad x \neq y.$$
This Markov chain admits a reversible measure given by~$\pi(x) = \deg(x)$.
\item Let~$G$ be a locally finite, connected simple graph and let~$c:E(G) \to (0,\infty)$ be a function that assigns to each edge~$e$ of~$G$ a positive \emph{conductance}~$c(e)$ that satisfies
$$ \sup_{x \in V(G)} \sum_{y: \{x,y\} \in E(G)} c(\{x,y\}) < \infty.$$ 
Consider the~$Q$-matrix
$$ q(x,y) = c(\{x,y\}) \cdot \mathds{1}\{ \{x,y\} \in E(G) \}, \quad x \neq y. $$
This Markov chain admits a reversible measure given by~$\pi(x)=1$. 
\end{enumerate}

\subsection{Martin capacities}
In this section, we recall the relation between Martin capacities and hitting probabilities established in~\cite{BenjaminiPemantlePeres1995}, and use it to compare the collision properties of discrete- and continuous-time Markov chains.

Let~$P$ be a transient irreducible Markov kernel on~$\mathsf S$, with~$n$-step transition probabilities~$p_n(x,y)=P^n(x,y)$. Its Green kernel is defined by
$$ G(x,y) = \sum^\infty_{n=0} p_n(x,y), x,y \in \mathsf S.$$
Fix a reference state~$\rho \in \mathsf S$. The Martin kernel associated with~$P$ and~$\rho$ is defined by
$$M(x,y) := \frac{G(x,y)}{G(\rho,y)}, \qquad x,y \in \mathsf S.$$
The capacity of~$\Lambda \subseteq \mathsf S$ with respect to the Martin kernel is
$$ \mathrm{Cap}_M(\Lambda) := \left[ \inf_{\mu} \sum_{x,y \in \Lambda} M(x,y) \mu(x) \mu(y)  \right]^{-1}, $$
where the infimum is over probability measures on~$\mathsf S$ supported in~$\Lambda$, with the convention~$\infty^{-1}=0$. The \emph{asymptotic capacity} of~$\Lambda$ is defined by
$$ \mathrm{Cap}^{(\infty)}_M(\Lambda) := \inf_{ \{ \Lambda_0 \text{ finite} \} } \mathrm{Cap}_M( \Lambda \setminus \Lambda_0 ). $$

Theorem~2.2 of~\cite{BenjaminiPemantlePeres1995} gives the following comparison between asymptotic Martin capacity and the probability of visiting a set infinitely often.
\begin{theorem}\label{thm:Martin_kernel}
Let~$(X_n)_{n \ge 0}$ be a transient Markov chain on~$\mathsf S$ with initial state~$\rho$, and let~$M$ be its Martin kernel. Then, for any~$\Lambda \subseteq S$,
$$ \frac 12 \cdot \mathrm{Cap}^{(\infty)}_M(\Lambda) \leq \Pp( X_n \in \Lambda \text{ infinitely often} ) \leq \mathrm{Cap}^{(\infty)}_M(\Lambda) $$
\end{theorem}

We now apply this result to the pair chains associated with a discrete-time Markov chain and its Poissonization. Let~$P$ be an irreducible Markov kernel on~$\mathsf S$ and consider the $Q$-matrix~$Q_P:= P - I$. Thus,
$$ q(x,y) = P(x,y), \quad x \neq y, \qquad \text{and} \qquad q(x) = 1 - P(x,x), \quad x \in \mathsf S. $$
The $Q_P$-chain is the rate-$1$ Poissonization of the discrete-time $P$-chain. Indeed, if~$(X_n)_{n\geq0}$ is a~$P$-chain and
$(N_t)_{t\geq0}$ is an independent Poisson process of rate~$1$, then~$(X_{N_t})_{t\geq0}$ is a continuous-time Markov chain with $Q$-matrix~$(P-I)$. Consequently, if~$P$ is reversible with respect to a measure~$\pi$, then~$Q_P$ is reversible with respect to the same measure.

Define the Markov kernels on $\mathsf S\times \mathsf S$
$$\mathsf P_\mathrm{disc}:=P\otimes P \qquad \text{and} \qquad \mathsf P_\mathrm{cont} := \frac{1}{2}\big(P \otimes I+I\otimes P \big). $$
The first kernel describes the joint evolution of two independent
discrete-time~$P$-chains. The second is the discrete-time kernel underlying the Poissonization of the pair of independent~$Q$-chains. Indeed, the generator of the continuous-time pair process is
\[ (P-I)\otimes I+I\otimes (P-I) = 2\big(\mathsf P_{\mathrm{cont}}-I\big). \]
Hence, the continuous-time pair process is the rate-$2$ Poissonization of the~$\mathsf P_{\mathrm{cont}}$-chain. 

Let
$$ \Delta:=\{(x,x):x\in S\} $$
be the diagonal of $\mathsf S\times \mathsf S$. Thus, visits of the~$\mathsf P_\mathrm{disc}$-chain (resp.~$\mathsf P_\mathrm{cont}$-chain) to~$\Delta$ correspond exactly to collisions of the two discrete-time (resp. continuous-time) chains. Recall the collision events from~\eqref{def_collision_event_disc} and~\eqref{def_collision_event_cont}. 

\begin{lemma}\label{lem:green_comparison_collision}
Let~$P$ be an irreducible Markov kernel on $\mathsf S$, and let $G_{\rm disc}$ and $G_{\rm cont}$ be the Green kernels associated with $\mathsf P_{\rm disc}$ and $\mathsf P_{\rm cont}$, respectively. Assume that there exist constants $0<c\leq C<\infty$ such that
\begin{equation}\label{ineq_doble_Green}
c \cdot G_{\rm disc}( (x,x),(y,y) ) \leq G_{\rm cont}( (x,x),(y,y) ) \leq C \cdot G_{\rm disc}( (x,x),(y,y) ), \quad x,y \in \mathsf S.
\end{equation}
Then, 
$$ \mathbb P \big( \mathcal E^P_\mathrm{disc} \big)>0 \; \text{ if and only if } \; \mathbb P ( \mathcal E^{Q_P}_\mathrm{cont} ) >0.$$
for every common initial state~$(\rho,\rho)$. 
\end{lemma}
\begin{remark}
If~$P$ has self-loops, visits of the~$\mathsf P_{\mathrm{cont}}$-chain to~$\Delta$ need not correspond one-to-one to distinct meetings. However, irreducibility implies that every consecutive run of self-loops on~$\Delta$ is almost surely. Hence, infinite visits to~$\Delta$ are equivalent to infinitely many continuous-time collisions.
\end{remark}

\begin{remark}\label{remark_01_upgrade}
Suppose, in addition, that~$P$ is reversible with respect to a measure~$\pi$ satisfying~$\sup_{x \in \mathsf S}\pi(x) < \infty$. Then, Proposition~\ref{prop_0-1_law} upgrades the conclusion to 
$$ \mathbb P \big( \mathcal E^P_\mathrm{disc} \big) = 1 \; \text{ if and only if } \; \mathbb P ( \mathcal E^{Q_P}_\mathrm{cont} ) = 1.$$
\end{remark}
\begin{proof}
Fix~$\rho\in\mathsf S$. By
\eqref{ineq_doble_Green},
\[
G_{\mathrm{disc}}( (\rho,\rho),(\rho,\rho))=\infty
\quad\Longleftrightarrow\quad
G_{\mathrm{cont}}((\rho,\rho),(\rho,\rho))=\infty.
\]
Suppose first that these quantities are infinite. Then~$(\rho,\rho)$ is recurrent for both pair chains. Consequently, both pair chains return to
$(\rho,\rho) \in\Delta$ infinitely often almost surely. Hence,
\[ \Pp \big(\mathcal E^P_{\mathrm{disc}}\big) = \Pp \big(\mathcal E^Q_{\mathrm{cont}}\big) = 1. \]

It remains to consider the transience case. By~\eqref{ineq_doble_Green},
\[
G_{\mathrm{disc}}( (\rho,\rho),(\rho,\rho)) < \infty
\quad\Longleftrightarrow\quad
G_{\mathrm{cont}}((\rho,\rho),(\rho,\rho)) < \infty.
\]
Hence, in this case, both pair chains are transient on the communicating class of~$(\rho,\rho)$. Since~$P$ is irreducible, every point of the diagonal~$\Delta$ belongs to this communicating class for both pair chains. We therefore apply Theorem~\ref{thm:Martin_kernel} on these classes, with reference state~$(\rho,\rho)$.

Let~$M_{\mathrm{disc}}$ and~$M_{\mathrm{cont}}$ denote the corresponding Martin kernels with reference state~$(\rho,\rho)$. For every~$x,y \in \mathsf S$, the Green-kernel comparison~\eqref{ineq_doble_Green} gives
\[ \frac{c}{C}\, M_{\mathrm{disc}}( (x,x),(y,y) )
\leq M_{\mathrm{cont}}( (x,x), (y,y) )
\leq \frac{C}{c}\, M_{\mathrm{disc}}( (x,x), (y,y) ). \]
Consequently,
\[
\frac{c}{C}\,
\operatorname{Cap}^{(\infty)}_{M_{\mathrm{disc}}}(\Delta)
\leq
\operatorname{Cap}^{(\infty)}_{M_{\mathrm{cont}}}(\Delta)
\leq
\frac{C}{c}\,
\operatorname{Cap}^{(\infty)}_{M_{\mathrm{disc}}}(\Delta).
\]
Then immediately:
\[
\operatorname{Cap}^{(\infty)}_{M_{\mathrm{disc}}}(\Delta)>0
\quad\Longleftrightarrow\quad
\operatorname{Cap}^{(\infty)}_{M_{\mathrm{cont}}}(\Delta)>0.
\]
The conclusion now follows from Theorem~\ref{thm:Martin_kernel} applied to
the diagonal~$\Delta$.
\end{proof}

\subsection{Proof of Theorem~\ref{thm:main}}\label{sec:proof_main}
The proof of the main theorem follows from the following more general Green-kernel comparison.
\begin{lemma}\label{lem:green-comparison}
Let~$P$ be an irreducible Markov kernel on $\mathsf S$, and let $G_{\rm disc}$ and $G_{\rm cont}$ be the Green kernels associated with $\mathsf P_{\rm disc}$ and $\mathsf P_{\rm cont}$, respectively. Suppose that, for some~$\alpha>0$,
\begin{equation}\label{condition_inf}
\inf_{x \in \mathsf S} P^2(x,x) \ge \alpha
\end{equation}
Then, there exists~$A = A(\alpha) > 0$ such that
$$ A \cdot G_{\rm disc}( (x,x);(y,y) ) \leq G_{\rm cont}( (x,x);(y,y) ) \leq 2 G_{\rm disc}( (x,x);(y,y) ). $$
for each~$x,y \in \mathsf S$.
\end{lemma}
We now have all the ingredients needed to prove the main theorem.

\begin{proof}[Proof of Theorem~\ref{thm:main}]
Let~$G$ be a bounded-degree connected graph, set~$\mathsf S=V(G)$, and let~$P$ be the transition kernel of a simple random walk on~$G$, that is,~$P(x,y) = \deg(x)^{-1} \cdot \mathds{1}\{ \{x,y\} \in E(G) \}$,~$x,y \in V(G)$. Set~$\alpha^{-1} = \sup_{x \in V(G)} \deg(x)$. We have that
$$ P^2(x,x) = \sum_{y} P(x,y) \cdot P(y,x) \ge \alpha \cdot \sum_y P(x,y) = \alpha. $$
Lemma~\ref{lem:green-comparison} then gives a two-sided comparison between the Green kernels~$G_{\rm disc}$ and~$G_{\rm cont}$ on the diagonal~$\Delta$. We now apply Lemma~\ref{lem:green_comparison_collision} to obtain that \[ \Pp\big(\mathcal E^P_{\rm disc}\big)>0 \quad\Longleftrightarrow\quad \Pp\big(\mathcal E^{Q_P}_{\rm cont}\big)>0. \] Moreover, simple random walk is reversible with respect to~$\pi(x)=\deg(x)$ and this measure is uniformly bounded since~$G$ has bounded degree. Remark~\ref{remark_01_upgrade} therefore upgrades the preceding equivalence to \[ \Pp\big(\mathcal E^P_{\rm disc}\big)=1 \quad\Longleftrightarrow\quad \Pp\big(\mathcal E^{Q_P}_{\rm cont}\big)=1. \] 
By Remark~\ref{remark_dichotomy_ICPFCP}, each collision event has probability either 0 or 1. This proves the result.
\end{proof}

\section{Proof of Lemma~\ref{lem:green-comparison}}\label{ss_lem:green-comparison}
For $n,k\geq 0$, define
\[ \Psi(n,k) := \binom{n+k}{n}2^{-(n+k)} = \Pp( \mathrm{Bin}(n+k,1/2) = n ). \]
Thus,~$\Psi(n,k)$ is the probability that, during the first~$n+k$ steps of the $\mathsf P_{\rm cont}$-chain, the first coordinate is updated exactly~$n$ times and the second coordinate exactly~$k$ times. Consequently,
\begin{align}
G_{\rm disc}((x,x);(y,y)) &= \sum_{n=0}^{\infty}p_n(x,y)^2, \label{def_G_disc} \\
G_\mathrm{cont}((x,x),(y,y)) &= \sum_{n,k=0}^{\infty}
\Psi(n,k)p_n(x,y)p_k(x,y). \label{def_G_cont}
\end{align}

The upper bound in Lemma~\ref{lem:green-comparison} is immediate from~\eqref{def_G_cont}. For the lower bound, we need the following estimate.

\begin{lemma}\label{lem:thickness}
Let~$P$ be an irreducible Markov kernel with~$n$-step transition probabilities~$p_n$, and suppose that, for some~$\alpha > 0$,
$$ \inf_{x \in \mathsf S} P^2(x,x) \ge \alpha.$$
Then, there exists~$a = a(\alpha) > 0$ such that
$$ \sum_{ \substack{\ell \ge 0: \\|\ell - m| \leq \sqrt{m}  } } \Psi(2m+\epsilon,2\ell + \epsilon) \cdot p_{2\ell + \epsilon}(x,y) \ge a \cdot p_{2m+\epsilon}(x,y) $$
for every~$x,y \in \mathsf S, m \ge 1$ and~$\epsilon \in \{0,1\}$.
\end{lemma}
We postpone the proof of Lemma~\ref{lem:thickness} and first complete the
proof of the Green-kernel comparison.

\begin{proof}[Proof of Lemma~\ref{lem:green-comparison}]
Fix~$x,y \in \mathsf S$. We begin with the upper bound. We have that
\begin{align*}
G_\mathrm{cont}( (x,x),(y,y)) & \stackrel{\eqref{def_G_cont}} \leq \frac12 \cdot
\sum_{n,k\geq 0} \Psi(n,k) \left(p_n(x,y)^2+p_k(x,y)^2 \right) \\[2mm]
&  \stackrel{(\ast)} = \sum_{n\geq 0}p_n(x,y)^2 + \sum_{k\geq 0}p_k(x,y)^2 = 2G_\mathrm{disc}((x,x),(y,y)),
\end{align*}
where the inequality follows from the inequality~$2ab \leq a^2 + b^2$ and the equality~$(\ast)$ from the symmetry of~$\Psi(n,k)$ and the identity~$\sum_{k \ge 0}\Psi(n,k) = 2$ for all~$n \ge 0$. This identity follows from the negative binomial series
$$ (1-s)^{-(n+1)} = \sum^\infty_{k=0} \binom{n+k}{k} s^k, \quad |s| < 1, $$
evaluated at~$s=1/2$. For the lower bound,
\begin{align*}
G_\mathrm{cont}((x,x),(y,y)) & = \sum_{n,k \ge 0} \Psi(n,k)p_n(x,y)p_k(x,y) \\[2mm]
& \ge \sum_{ \epsilon \in \{0,1\} } \sum_{m \ge 1} \sum_{ \substack{\ell \ge 0: \\ |\ell-m| \leq \sqrt{m} } }\Psi(2m+\epsilon,2\ell+\epsilon) \cdot p_{2m+\epsilon}(x,y)p_{2\ell+\epsilon}(x,y) \\[2mm]
& \stackrel{(\ast)} \ge \sum_{ \epsilon \in \{0,1\} } \sum_{m \ge 1} a \cdot p_{2m+\epsilon}(x,y)^2 = \sum_{n \ge 2} a \cdot p_n(x,y)^2 \ge \frac{a \alpha^2}{1+\alpha^2} \cdot G_\mathrm{disc}((x,x),(y,y)),
\end{align*}
where~$(\ast)$ follows from Lemma~\ref{lem:thickness} and the last inequality from Chapman-Kolmogorov identity:
$$ \sum_{n \ge 2} \cdot p_{n}(x,y)^2 \ge p_2(x,y)^2 + p_3(x,y)^2 \stackrel{\eqref{condition_inf}} \ge \alpha^2 \cdot \big( p_0(x,y)^2 + p_1(x,y)^2 \big). $$
We conclude the proof by taking~$A = \frac{a \alpha^2}{1+\alpha^2}$.
\end{proof}

\subsection{Binomial estimates}
We now prove Lemma~\ref{lem:thickness}. We use the following standard local
binomial estimate. Fix~$p \in (0,1)$ and~$L < \infty$. There exists constants~$0 < c_{p,L} \leq C_{p,L} < \infty$ such that
$$ \frac{c_{p,L}}{\sqrt{N}} \leq \Pp( \mathrm{Bin}( N,p ) = k ) \leq \frac{C_{p,L} }{\sqrt{N}} $$
whenever~$|k - Np| \leq L \sqrt{N}$; see, e.g., \cite[Chapter VII, Theorem 1]{Petrov1975}. We will use two consequences of this estimate.

First, there exists $a_1>0$ such that
\begin{equation}\label{ineq_bin_1}
\Psi(2m+\epsilon,2\ell + \epsilon) = \Pp( \mathrm{Bin}(2m+2\ell + 2\epsilon,1/2)=2\ell+\epsilon) \ge \frac{a_1}{\sqrt{m}}
\end{equation}
for each~$\epsilon \in \{0,1\}$,~$m \ge 1$ and~$\ell \ge 0$ with~$|m-\ell| \leq\sqrt{m}$.

\begin{proof}
Take~$N = 2m+2\ell+2\epsilon$ and~$k = 2\ell + \epsilon$ with parameter~$p = 1/2$. Then,
$$ |k-Np| = |\ell - m | \leq \sqrt{m} \leq \sqrt{N}. $$
The local binomial estimate therefore gives
$$ \Pp( \mathrm{Bin}(2m+2\ell + 2\epsilon,1/2)=2\ell+\epsilon) \ge \frac{c_{0.5,1}}{\sqrt{2m+2\ell+2\epsilon}} \ge \frac{c_{0.5,1}}{\sqrt{8m}}, $$
where the last inequality follows from the hypothesis relationship, which implies~$\ell \leq \sqrt{m} + m$. We complete the proof by setting~$a_1 = c_{0.5,1}/\sqrt{8}$.    
\end{proof}

Our second estimate is the following. For every~$p \in (0,1)$, there exists~$a_2 = a_2(p)>0$ such that 
\begin{equation}\label{ineq_bin_2}
\sum_{ \substack{\ell \ge 0: \\|\ell - m| \leq \sqrt{m} } }  \Pp(\mathrm{Bin}(\ell,p)=k ) \ge a_2 \sqrt{m} \cdot \Pp(\mathrm{Bin}(m,p)=k)
\end{equation}
for each~$m \ge 1$ and~$k \ge 0$.

\begin{proof}
Fix~$m \ge 1$ and~$k \ge 0$. If~$k > m$, then the probability on the right-hand side of~\eqref{ineq_bin_2} is zero, so it holds trivially. We may therefore assume that~$k \le m$. We define
$$ R(\ell) := \Pp( \mathrm{Bin}(\ell,p)=k ), \quad \ell \ge k. $$
We first locate its mode. A direct computation gives
$$ \frac{R(\ell+1)}{R(\ell)} = (1-p) \cdot \frac{\ell+1}{\ell+1-k} \ge 1 \quad \text{if and only if} \quad p(1+\ell) \leq k. $$
Then, the sequence~$R(\ell)$ increases until~$M:= \lfloor k/p \rfloor$ and decreases afterwards. We divide the analysis into two cases:

Suppose~$|m-M| \ge \lfloor \sqrt{m} \rfloor$. If~$m < M$, then by unimodality
$$ R(\ell) \ge R(m) \quad \text{for every} \quad \ell \in \{m,\dots,m+\lfloor \sqrt{m} \rfloor \}.$$
Likewise for~$m > M$. Thus, in either case there are~$\lfloor \sqrt{m} \rfloor+1$ integers~$\ell$ satisfying~$|\ell-m| \leq \sqrt{m}$ and~$R(\ell) \ge R(m)$. Then,
$$ \sum_{ \substack{\ell \ge 0: \\|\ell - m| \leq \sqrt{m} } }  R(\ell) \ge \sqrt{m} \cdot R(m). $$

Now, suppose~$|m-M| \leq \sqrt{m}$. Note that the definition of~$M$ implies~$|k-pM|<p \leq \sqrt{m}$. Then, these two inequalities imply 
\begin{equation}\label{eq_final_1}
|k-pm| \leq |k-pM| + p|m-M| \leq 2 \sqrt{m}
\end{equation}
The upper bound in the local binomial estimate gives
\begin{equation}\label{eq_final_2}
\sqrt{m} \cdot R(m) \leq C_{p,2}.
\end{equation}
Now, let~$\ell \in \{m,\dots,m+\lfloor \sqrt{m} \rfloor\}$. By~\eqref{eq_final_1},
$$ |k-p\ell| \leq |k-pm| + p|\ell-m| \leq 3\sqrt{m} \leq 3 \sqrt{\ell}. $$
Thus the lower bound in the local binomial estimate yields
$$ R(\ell) \ge \frac{c_{p,3}}{\sqrt{\ell}} \ge \frac{c_{p,3}}{\sqrt{2m}} \quad \text{for every} \quad \ell \in \{m,\dots,m+\lfloor \sqrt{m} \rfloor \}. $$
Hence,
$$ \sum_{ \substack{\ell \ge 0: \\|\ell - m| \leq \sqrt{m} } } R(\ell) \ge \sum^{m+\lfloor \sqrt{m} \rfloor}_{\ell = m} R(\ell) \ge \frac{c_{p,3}}{\sqrt{2}} \stackrel{\eqref{eq_final_2}} \ge \frac{c_{p,3}}{\sqrt{2} C_{p,2}} \sqrt{m} \cdot R(m) . $$
Therefore, combining the two cases,~\eqref{ineq_bin_2} holds with~$a_2 = \min\{1,\frac{c_{p,3}}{\sqrt{2}C_{p,2}}\}$. This completes the proof. 
\end{proof}

We are now ready to prove Lemma~\ref{lem:thickness}.
\begin{proof}[Proof of Lemma~\ref{lem:thickness}]
If~$\alpha=1$, irreducibility implies that~$P$ is the deterministic two-cycle (apart from the trivial one-state case), for which the claim is immediate. Hence, we may assume~$\alpha \in (0,1)$ and write
$$ P^2 = \alpha I + (1-\alpha) S, $$
where
$$ S = \frac{P^2-\alpha I}{1-\alpha} $$
is a Markov kernel that commutes with~$P$. Fix~$\epsilon \in \{0,1\}$. For every~$\ell \ge 0$ and~$x,y \in \mathsf S$,
$$ p_{2\ell+\epsilon}(x,y) = \bigl [(P^2)^\ell P^\varepsilon\bigr ](x,y) = \sum_{k=0}^\ell \binom{\ell}{k} \alpha^{\ell-k}(1-\alpha)^k (S^kP^\epsilon)(x,y).$$
Hence, for~$m \ge 1$,
\begin{align*}
\sum_{ \substack{\ell \ge 0: \\|\ell - m| \leq \sqrt{m} } } \Psi(2m+\epsilon,2\ell+\epsilon) \cdot p_{2\ell + \epsilon}(x,y) & \stackrel{\eqref{ineq_bin_1}} \ge \frac{a_1}{\sqrt{m}} \cdot \sum_{ \substack{\ell \ge 0: \\|\ell - m| \leq \sqrt{m} } } \sum^\ell_{k=0} \Pp(\mathrm{Bin}(\ell,1-\alpha)=k ) \cdot (S^kP^\epsilon)(x,y) \\[2mm]
& = \frac{a_1}{\sqrt{m}} \cdot \sum^\infty_{k=0} (S^kP^\epsilon)(x,y) \cdot \sum_{ \substack{\ell \ge 0: \\|\ell - m| \leq \sqrt{m} } }  \Pp(\mathrm{Bin}(\ell,1-\alpha)=k ) \\[2mm]
& \stackrel{\eqref{ineq_bin_2}} \ge a_1 a_2 \cdot \sum^\infty_{k=0} \Pp(\mathrm{Bin}(m,1-\alpha)=k) \cdot (S^kP^\epsilon)(x,y) \\[2mm]
& = a_1a_2 \cdot p_{2m+\epsilon}(x,y),
\end{align*}
where the first equality follows from Tonelli's theorem. We conclude the proof by setting~$a = a_1 a_2$.
\end{proof}

\section{Further time parametrizations}\label{ss_further}
The preceding argument relies essentially on the constant-speed parametrization. More generally, suppose that a continuous-time chain has embedded jump kernel~$P$ and total jump rate~$q(x)>0$ at~$x$. The embedded jump chain of the pair process has transition kernel
\[ \mathsf P_q\big((x,y),(x',y)\big) = \frac{q(x)}{q(x)+q(y)}P(x,x'),
\qquad \mathsf P_q\big((x,y),(x,y')\big) = \frac{q(y)}{q(x)+q(y)}P(y,y'). \]
where the coefficients depend on the current state~$(x,y)$. When~$q$ is constant, these coefficients are both~$1/2$, so the updates of the two coordinates are interlaced according to an independent binomial law. This geometry-independent feature is precisely what is used in the proof of Lemma~\ref{lem:green-comparison}; for nonconstant~$q$, it is no longer available.

There is nevertheless a useful comparison at the level of recurrence. If~$P$ is reversible with respect to~$\pi$ and
\[ 0<\inf_{x\in\mathsf S}q(x)\leq \sup_{x\in\mathsf S}q(x)<\infty. \]
The kernel~$\mathsf P_q$ is reversible with respect to
$$ \pi_q(x,y) := \pi(x) \pi(y) \cdot \frac{ q(x) + q(y) }{q(x)q(y)}.$$
Moreover, the conductances associated with~$\mathsf P_q$ are uniformly comparable with those of the constant-speed pair chain~$\mathsf P_{\rm cont}$. Hence, the two pair chains have the same recurrence/transience classification. In particular, if they are recurrent, then, starting from a diagonal state, both return to that state infinitely often almost surely and therefore yield infinitely many collisions.

A natural example is the variable-speed simple random walk, for which~$q(x)=\deg(x)$. On a connected bounded-degree graph, the preceding recurrence comparison applies. Thus, only the transient case remains beyond the scope of the present argument. A uniform pointwise comparison of the corresponding Green kernels does not appear to be a natural property in this generality. We therefore do not pursue this case here; addressing it seems to require ideas different from the Green-kernel comparison developed in this work.

\subsection*{Acknowledgments}
J.A. was supported by FAPESP (grants 2023/13453-5 and 2025/02707-1); the author is thankful for this support.

\bibliographystyle{alpha}
\bibliography{ref.bib}

\end{document}